\pdfoutput=1
\documentclass{article}
\usepackage{amsmath,amssymb,amsthm}
\usepackage[english]{babel}
\usepackage[T1]{fontenc}
\usepackage{newtxtext,newtxmath}
\usepackage[utf8]{inputenc}
\usepackage[babel=true]{microtype}
\usepackage[autostyle=true]{csquotes}

\usepackage{slashed}
\usepackage{mathrsfs}
\newcommand{\textprime}{\textsuperscript{\ensuremath{\prime}}}
\usepackage{thmtools,thm-restate}
\usepackage[pdfusetitle]{hyperref}
\hypersetup{
  colorlinks=true,
  allcolors=blue,
}

\usepackage[citestyle=alphabetic,bibstyle=alphabetic,backend=biber,url=true,doi=true,isbn=false,maxbibnames=99,firstinits=true]{biblatex}

\AtEveryBibitem{\clearfield{month}}
\AtEveryBibitem{\clearfield{day}}
\AtEveryBibitem{\clearfield{series}}
\AtEveryBibitem{\clearfield{language}}
\AtEveryBibitem{
  \ifentrytype{book}
    {\clearfield{pages}
     \clearfield{volume}}
    {}
}
\AtEveryBibitem{
  \ifentrytype{article}
    {\clearfield{url}}
    {}
}
\bibliography{literature.bib}
\usepackage{tikz-cd}
\usepackage{tikz}
\usepackage[shortlabels]{enumitem}
 \usepackage[final]{fixme}
 \fxsetup{layout={margin}}
\declaretheorem[name=Theorem, numberwithin=section]{thm}
\declaretheorem[name=Lemma,numberlike=thm]{lem}
\declaretheorem[name=Corollary,numberlike=thm]{cor}
\declaretheorem[name=Proposition,numberlike=thm]{prop}
\declaretheorem[name=Definition,numberlike=thm, style=definition]{defi}
\declaretheorem[name=Geometric Setup,numberlike=thm, style=definition]{setup}
\declaretheorem[name=Example, numberlike=thm, style=remark]{ex}
\declaretheorem[name=Remark, numberlike=thm, style=remark]{rem}

\allowdisplaybreaks[1]
\usepackage[noabbrev]{cleveref} 
\crefname{figure}{Figure}{Figures}
\crefname{table}{Table}{Tables}
\crefname{thm}{Theorem}{Theorems}
\crefname{lem}{Lemma}{Lemmas}
\crefname{defi}{Definition}{Definitions}
\crefname{cor}{Corollary}{Corollaries}
\crefname{prop}{Proposition}{Propositions}
\crefname{ex}{Example}{Examples}
\crefname{rem}{Remark}{Remarks}
\crefname{section}{Section}{Sections}
\crefname{chapter}{Chapter}{Chapters}
\crefname{appendix}{Appendix}{Appendices}
\crefname{conj}{Conjecture}{Conjectures}
\crefname{setup}{Geometric Setup}{Geometric Setup}
\creflabelformat{enumi}{#2#1#3}
\newcommand{\parensup}[1]{\textup{(}#1\textup{)}}

\newcommand{\tens}{\otimes}

\newcommand{\agrtens}{\widehat{\agrtens}}

\newcommand{\field}[1]{\mathbb{#1}}

\newcommand{\N}{\field{N}}
\newcommand{\Q}{\field{Q}}
\newcommand{\R}{\field{R}}

\newcommand{\Rgeq}{\R_{\geq 0}}
\newcommand{\Rleq}{\R_{\leq 0}}

\newcommand{\Cstar}{\mathrm{C}^*}
\newcommand{\cont}{\mathcal{C}}
\newcommand{\contZ}{\cont_0}
\newcommand{\pr}{\operatorname{pr}}

\newcommand{\bd}{\partial}
\newcommand{\bdMV}{\bd_{\mathrm{MV}}}
\newcommand{\id}{\operatorname{id}}

\newcommand{\iso}{\cong}

\newcommand{\Ind}{\operatorname{Ind}}

\newcommand{\roeAlg}{\mathrm{C}^*}

\newcommand{\KTh}{\mathrm{K}}

\newcommand{\KPh}{\ast}

\newcommand{\homlgy}{\mathrm{H}}
\DeclareMathOperator*{\chern}{ch}

\newcommand{\Ahat}{\hat{\mathrm{A}}}

\newcommand{\tangentBdl}{\mathrm{T}}

\newcommand{\diracOp}{\slashed{D}}

\newcommand{\Efree}{\mathrm{E}}
\newcommand{\Bfree}{\mathrm{B}}
\DeclareMathOperator*{\colim}{colim}
\newcommand{\redd}{\mathrm{red}}

\renewcommand{\leq}{\leqslant}
\renewcommand{\geq}{\geqslant}

\newcommand{\asdim}{\operatorname{asdim}}
\newcommand{\lf}{\mathrm{lf}}
\newcommand{\pt}{\mathrm{pt}}

\newcommand{\IndRcptl}{\mathcal{C}}

\newcommand{\PM}{\mathrm{PM}}
\newcommand{\Klf}{\KTh^\lf}
\title{An index obstruction to positive scalar curvature on fiber bundles over aspherical manifolds}
\author{Rudolf Zeidler}
\date{}
\newcommand{\Addresses}{{
  \bigskip
  \footnotesize
  \textsc{Mathematisches Institut, WWU Münster, Einsteinstrasse 62, 48149 Münster, Germany}\par\nopagebreak
  \textit{Email address}: \href{mailto:math@rzeidler.eu}{math@rzeidler.eu}
  }}

\begin{document}
  \maketitle
  \begin{abstract}
     We exhibit geometric situations, where higher indices of the spinor Dirac operator on a spin manifold $N$ are obstructions to positive scalar curvature on an ambient manifold $M$ that contains $N$ as a submanifold.
     In the main result of this note, we show that the Rosenberg index of $N$ is an obstruction to positive scalar curvature on $M$ if $N \hookrightarrow M \twoheadrightarrow B$ is a fiber bundle of spin manifolds with $B$ aspherical and $\pi_1(B)$ of finite asymptotic dimension.
     The proof is based on a new variant of the multi-partitioned manifold index theorem which might be of independent interest.
     Moreover, we present an analogous statement for codimension one submanifolds.
     We also discuss some elementary obstructions using the $\Ahat$-genus of certain submanifolds.
  \end{abstract}
  \section{Introduction}
  We consider the following setup:
    \begin{setup}\label{defi:setup}
    Let $M$ be a closed connected spin $m$-manifold and $N \subseteq M$ a closed connected submanifold of codimension $q$ with trivial normal bundle.
  Moreover, we denote the fundamental groups of $M$ and $N$ by $\Gamma$ and $\Lambda$, respectively, and let $j \colon \Lambda \to \Gamma$ be the map induced by the inclusion $\iota \colon N \hookrightarrow M$.
  \end{setup}
    Hanke--Pape--Schick~\cite{HPS14Codimension} have found that if the codimension $q = 2$ and some assumptions on homotopy groups hold, then the Rosenberg index of $N$ is an obstruction to positive scalar curvature on $M$.
  Motivated by this result, it is an interesting endeavor to find further situations where the Rosenberg index of $N$ is an obstruction to positive scalar curvature on the ambient manifold $M$.
  In this note, we exhibit certain cases where it is possible to relax the restrictions on the codimension.

  Recall the \emph{Rosenberg index} $\alpha^{\Gamma}(M) \in \KTh_\ast(\Cstar_\epsilon \Gamma)$ of a closed spin manifold $M$, where $\Gamma = \pi_1(M)$ and $\Cstar_\epsilon \Gamma$ denotes the maximal ($\epsilon = \max$) or reduced ($\epsilon = \redd$) group $\Cstar$-algebra.
  Abstractly, it is obtained by applying the Baum--Connes assembly map
  \begin{equation*}
    \mu \colon \KTh_\ast(\Bfree \Gamma) \to \KTh_\ast(\Cstar_\epsilon \Gamma),
  \end{equation*}
  to the image of the $\KTh$-homological fundamental class of $M$ under the map $u \colon M \to \Bfree \Gamma$ that classifies the universal covering.
  The \emph{\parensup{maximal, if $\epsilon = \max$} strong Novikov conjecture} predicts that $\mu \tens \Q$ is injective.

  All statements in the introduction are made under implicit assumption of \cref{defi:setup}.
  We start by recalling the precise statement of Hanke--Pape--Schick's codimension two obstruction.

  \begin{thm}[{\cite[Theorem 1.1]{HPS14Codimension}}]\label{thm:HPSCodim2}
    Let $\epsilon \in \{ \redd, \max \}$.
    Let $N$ have codimension $q = 2$ and suppose that $j \colon \Lambda \to \Gamma$ is injective as well as $\pi_2(M) = 0$.
    If $\alpha^\Lambda(N)\neq 0 \in \KTh_{m-2}(\Cstar_\epsilon \Lambda)$, then $M$ does not admit a metric of positive scalar curvature.
  \end{thm}
  \begin{rem}\label{rem:indexNonZero}
    The theorem was proved by applying methods from Roe's coarse index theory to a manifold that arises out of a modification of a certain covering of $M$.
    As discussed in~\cite[3]{HPS14Codimension}, this proof only shows that $M$ does not admit positive scalar curvature and does not give $\alpha^\Gamma(M) \neq 0$.
    However, the theorem actually implies that $M$ stably does not admit positive scalar curvature and hence non-vanishing of $\alpha^\Gamma(M)$ would be a consequence of the stable Gromov--Lawson--Rosenberg conjecture (at least if we worked with real $\KTh$-theory throughout).
    It is an open question whether it is possible to prove non-vanishing of $\alpha^\Gamma(M)$ directly under the hypotheses of \cref{thm:HPSCodim2}.
   \end{rem}
  \subsection{Obstructions on fiber bundles and codimension one} Hanke--Pape--Schick state the following application of \cref{thm:HPSCodim2} to fiber bundles:
  \begin{cor}[{\cite[Corollary 4.5]{HPS14Codimension}}]\label{thm:codim2FibreBundle}
    Let $\epsilon \in \{\redd, \max\}$.
    Suppose that $N \hookrightarrow M \twoheadrightarrow \Sigma$ is a fiber bundle, where $\pi_2(N) = 0$ and $\Sigma$ is a closed surface different from $S^2, \R P^2$.
    If $\alpha^\Lambda(N)\neq 0 \in \KTh_{m-2}(\Cstar_\epsilon \Lambda)$, then $M$ does not admit a metric of positive scalar curvature.
  \end{cor}
  In this special case it turns out that we can settle the question from \cref{rem:indexNonZero} by a more direct argument.
  Indeed, in the following main result of this note, we generalize \cref{thm:codim2FibreBundle} to arbitrary dimensions of the base manifold as well as obtain the stronger conclusion of non-vanishing of $\alpha^\Gamma(M)$:
  \begin{thm}\label{thm:fibreBundleObstruction}
    Let $\epsilon \in \{ \redd, \max \}$.
    Suppose that $N \overset{\iota}{\hookrightarrow} M \overset{\pi}{\twoheadrightarrow} B$ is a fiber bundle, where $B$ is aspherical and $\pi_1(B) = \Gamma / \Lambda$ has finite asymptotic dimension.
    If $\alpha^\Lambda(N) \neq 0 \in \KTh_{m-q}(\Cstar_\epsilon \Lambda)$, then $\alpha^\Gamma(M) \neq 0 \in\KTh_{m}(\Cstar_\epsilon \Gamma)$.
    In particular, $M$ does not admit positive scalar curvature in this case.
  \end{thm}
  In the proof we also employ coarse index theory.
  More specifically, we apply the multi-partitioned manifold index theorem.
  Although variants of it have been established previously by Siegel~\cite{Siegel12Mayer} and Schick--Zadeh~\cite{SZ13Multi}, neither of these references provides the theorem in sufficient generality for our purposes.
  Thus, in~\cref{sec:multipartitioned}, we have included a concise proof of the required result which might be of independent interest (see \cref{thm:multiPartitioned}).

  Unlike the case of \cref{thm:HPSCodim2}, in the proof of \cref{thm:fibreBundleObstruction} we apply the $q$-multi-partitioned manifold index theorem directly to a covering of $M$ (without modifying it further) and thereby obtain the stronger conclusion of $\alpha^\Gamma(M) \neq 0$.
  If $B$ is a surface or, more generally, admits non-positive sectional curvature, then the fact that a suitable covering of $M$ is $q$-partitioned follows from the Cartan--Hadamard theorem applied to $B$.
  To obtain the level of generality as stated, we apply a result of Dranishnikov which says that an aspherical manifolds with a fundamental group of finite asymptotic dimension has a stably hypereuclidean universal covering~\cite{Dranishnikov06hypereuclidean}.
  \begin{rem}\label{rem:HPSmoreGeneral}
  Unlike \cref{thm:codim2FibreBundle}, the condition $\pi_2(N) = 0$ is not required by \cref{thm:fibreBundleObstruction}.
  \fxnote{That was some time ago when I checked it, so maybe convince myself once again...}This, however, is not just a feature of our method: In fact, a careful reading of the proof from \cite{HPS14Codimension} reveals that in \cref{thm:HPSCodim2} the hypothesis of $\pi_2(M) = 0$ can be weakened to surjectivity of the map $\pi_2(N) \to \pi_2(M)$.
  \end{rem}
  Moreover, the idea for \cref{thm:fibreBundleObstruction} works even in full generality in codimension one (without assumptions on higher homotopy groups or on being a fiber bundle):
  \begin{thm}\label{thm:codim1Obstruction}
    Let $\epsilon \in \{ \redd, \max \}$.
    Let $N$ have codimension $q = 1$ and suppose that $j \colon \Lambda \to \Gamma$ is injective.
    If $\alpha^\Lambda(N) \neq 0 \in \KTh_{m-1}(\Cstar_\epsilon \Lambda)$, then $\alpha^\Gamma(M) \neq 0 \in\KTh_{m}(\Cstar_\epsilon \Gamma)$.
    In particular, $M$ does not admit positive scalar curvature in this case.
  \end{thm}
    \begin{rem}
    In the proofs of \cref{thm:fibreBundleObstruction,thm:codim1Obstruction}, a homomorphism $\Psi \colon \KTh_\ast(\Cstar_\epsilon \Gamma) \to \KTh_{\ast-q}(\Cstar_\epsilon \Lambda)$ with $\Psi(\alpha^\Gamma(M)) =\alpha^\Lambda(N)$ is constructed which might be of independent interest.
  \end{rem}

  \subsection{Higher \texorpdfstring{$\Ahat$}{A-hat} obstructions via submanifolds}
  In addition to our result on fiber bundles, we have some obstructions via the $\Ahat$-genus of submanifolds of arbitrary codimension under some restriction on the homotopy groups.
  In contrast to before, the proofs of the results below do not employ coarse index theory and essentially only rely on elementary techniques from (co)homology theory.

 First we state a result that applies to intersections of codimension two submanifolds.
  We continue to work in \cref{defi:setup}.
  \begin{thm}\label{thm:codim2IntersectionObstruction}
    Let $N = N_1 \cap \ldots \cap N_k$, where $N_1, \ldots, N_k \subseteq M$ are closed submanifolds that intersect mutually transversely and have trivial normal bundles.
    Suppose that the codimension of $N_i$ is at most two for all $i \in \{1, \ldots, k\}$ and that $\pi_2(N) \to \pi_2(M)$ is surjective.

     If $\Ahat(N) \neq 0$, then $\alpha^\Gamma(M) \neq 0 \in \KTh_\ast(\Cstar_{\max} \Gamma)$.
     In particular, $M$ does not admit a metric of positive scalar curvature in this case.
  \end{thm}
  In particular, specializing to a single codimension two submanifold, this settles the question of \cref{rem:indexNonZero} in the case of $\Ahat(N) \neq 0$ (which implies $\alpha^\Lambda(N) \neq 0$).

  The proof of this theorem (see \cref{sec:proofOfCohomological}) proceeds as follows: First we show that the surjectivity of $\pi_2(N) \to \pi_2(M)$ allows to rewrite $\Ahat(N)$ as a higher $\Ahat$-genus of $M$.
   Afterwards we appeal to a result of Hanke--Schick~\cite{HS08Low} about the maximal strong Novikov conjecture in low cohomological degrees and conclude that $\alpha^\Gamma(M) \neq 0 \in \KTh_\ast(\Cstar_{\max} \Gamma)$.
  If we allow higher codimensions for the submanifolds $N_i$, our method still works but we are no longer in a position to apply~\cite{HS08Low} and hence need to assume the strong Novikov conjecture:

  \begin{thm}\label{thm:generalIntersectionObstruction}
   Let $\epsilon \in \{ \redd, \max \}$.
   Let $N = N_1 \cap \ldots \cap N_k$, where $N_1, \ldots, N_k \subseteq M$ are closed submanifolds that intersect mutually transversely and have trivial normal bundles.
   Let $d$ be the maximum of the codimensions of $N_i$ over all $i \in \{1, \ldots, k\}$ and suppose that $\pi_j(M) = 0$ for $2 \leq j \leq d$.

   If $\Ahat(N) \neq 0$ and $\Gamma$ satisfies the (maximal, if $\epsilon = \max$) strong Novikov conjecture, then $\alpha^\Gamma(M) \neq 0 \in \KTh_\ast(\Cstar_\epsilon \Gamma)$.
  \end{thm}
  Note that the conditions on the homotopy groups are also slightly different than in \cref{thm:codim2IntersectionObstruction}.
  In fact, in \cref{thm:cohomologicalObstruction}, we prove our results under a more general homological condition which includes the restrictions on the homotopy groups from \cref{thm:codim2IntersectionObstruction,thm:generalIntersectionObstruction} as a special case (see \cref{thm:translatingHomotopyConditions}).

If we restrict \cref{thm:generalIntersectionObstruction} to a single submanifold, we obtain:
  \begin{cor}\label{thm:cohomologicalHigherCodimSpecial}
   Let $\epsilon \in \{ \redd, \max \}$.
    Suppose that $N$ has codimension $q$ and $\pi_j(M) = 0$ for $2 \leq j \leq q$.
    If $\Ahat(N) \neq 0$ and $\Gamma$ satisfies the (maximal, if $\epsilon = \max$) strong Novikov conjecture, then $\alpha^\Gamma(M) \neq 0 \in \KTh_\ast(\Cstar_\epsilon \Gamma)$.
  \end{cor}
  In the special case that $\Gamma$ is virtually nilpotent (which implies the strong Novikov conjecture), the consequence of \cref{thm:cohomologicalHigherCodimSpecial} that $M$ cannot admit positive scalar curvature was proved by Engel using a different method~\cite{Engel15Rough}.

  Moreover, under the assumptions of \cref{thm:cohomologicalHigherCodimSpecial}, even higher $\Ahat$-genera of $N$ are obstructions to psc on $M$.
  This was also discovered by Engel using yet a different method, see \cite{Engel16Localizing}.

  \section{The multi-partitioned manifold index theorem}\label{sec:multipartitioned}
  \subsection{Coarse index theory}
   Here we briefly review the relevant aspects of coarse index theory, see~\cite{Roe96Index,HR00Analytic}.
   Let $\epsilon \in \{ \redd, \max \}$ be fixed in this section.
   Let $X$ be a proper metric space endowed with an isometric, free and proper action of a discrete group $\Gamma$.
   We denote the \emph{$\Gamma$-equivariant Roe algebra} of $X$ by $\roeAlg(X)^\Gamma$.
   It is defined to be the (spacial if $\epsilon = \redd$ or maximal if $\epsilon = \max$) completion of the $\ast$-algebra of all $\Gamma$-equivariant locally compact operators of finite propagation defined over a fixed suitable Hilbert space representation of $\contZ(X)$.
   Recall the \emph{index map} (or \emph{assembly map}) from locally finite $\KTh$-homology of the quotient $\Gamma \backslash X$ to the $\KTh$-theory of equivariant Roe algebra:
   \begin{equation}
   \Ind^\Gamma \colon \Klf_\ast(\Gamma \backslash X) \to \KTh_\ast(\roeAlg(X)^\Gamma).\label{eq:indexMap}
   \end{equation}
   For an explicit definition of the assembly in the non-equivariant case (also pertaining to $\epsilon = \max$), see for instance~\cite[Subsection~4.6]{GWY08Novikov}.
   A straightforward generalization of the same formulas to the equivariant case then yields the equivariant assembly map $\KTh^{\mathrm{lf},\Gamma}_\ast(X) \to \KTh_\ast(\roeAlg(X)^\Gamma)$.
   To obtain the map as displayed in \labelcref{eq:indexMap}, we precompose with the induction isomorphism $\Klf_\ast(\Gamma \backslash X) \iso \KTh^{\mathrm{lf},\Gamma}_\ast(X)$ in analytic K-homology as it is exhibited via the Paschke duality picture in~\cite[Lemma~12.5.4]{HR00Analytic}, \cite[Theorem~4.3.25]{Siegel12Homological}.

   If $X$ is a complete spin $m$-manifold, we may apply the index map to the class $[\diracOp_{\Gamma \backslash X}] \in \Klf_m(\Gamma \backslash X)$ of the spinor Dirac operator on $\Gamma\backslash X$.
   We will use the notation $\Ind^\Gamma(\diracOp_{X}) := \Ind^\Gamma([\diracOp_{\Gamma \backslash X}])$.
   If $X = \tilde{M}$ is the universal covering of a closed spin manifold $M$ and $\Gamma = \pi_1(M)$, then there is a canonical isomorphism $\KTh_\ast(\roeAlg(\tilde{M})^\Gamma) \iso \KTh_\ast(\Cstar_\epsilon \Gamma)$ and $\Ind^\Gamma(\diracOp_{\tilde{M}})$ recovers the Rosenberg index $\alpha^\Gamma(M)$.\fxnote{Perhaps add some references for all of this.}

   In the following we introduce some notation which will feature in our formulation of the multi-partitioned manifold index theorem.
  Let $\Gamma$ be a countable discrete group and fix a model for the classifying space $\Bfree \Gamma$ as a locally finite simplicial complex.
  As usual, we denote its universal covering by $\Efree \Gamma$.
    \begin{defi}
    Let $Y$ be a proper metric space and define
    \begin{gather*}
      \Gamma \Klf_i(Y) := \colim_{Z}\KTh_i^\lf(Z),\\
      \Gamma \IndRcptl_i(Y) := \colim_{Z}\KTh_i(\roeAlg(\tilde{Z})^\Gamma),
    \end{gather*}
    where the colimits range over \emph{admissible subsets} $Z \subseteq \Bfree \Gamma \times Y$ and $Z$ is called admissible if it is closed and $\pr_2|_Z \colon Z \to Y$ is proper.
    Moreover, $\tilde{Z}$ denotes the lift of $Z$ to $\Efree \Gamma \times Y$.
  \end{defi}
  Roughly speaking, $\Gamma \Klf_i(Y)$ behaves like locally finite $\KTh$-homology in $Y$ and like ordinary $\KTh$-homology in the \enquote{$\Bfree \Gamma$-slot}.

  Recall that a map $f \colon (Y, d) \to (Y^\prime, d^\prime)$ between metric spaces is called \emph{coarse} if $f^{-1}(B^\prime)$ is bounded for each bounded set $B^\prime \subseteq Y^\prime$ and there exists a function $\rho \colon \Rgeq \to \Rgeq$ such that $d^\prime(f(x), f(y)) \leq \rho(d(x,y))$ for all $x,y \in Y$.
  Since the $\KTh$-theory of the (equivariant) Roe algebra is functorial with respect to (equivariant) coarse maps~\cite[Definition~6.3.15]{HR00Analytic}, the group $\Gamma \IndRcptl_i(Y)$ is functorial in $Y$ with respect to coarse maps.

  The index map~\labelcref{eq:indexMap} induces an index map in the limit $\Ind^\Gamma \colon \Gamma \Klf_\KPh(Y) \to \Gamma \IndRcptl_\KPh(Y)$ which is natural in $Y$ with respect to continuous coarse maps.

    \begin{ex}
    Taking $Y = \pt$ to be a point, $\Gamma \KTh_\KPh(\pt) = \KTh_\ast(\Bfree \Gamma)$ as defined via the $\KTh$-theory spectrum and $\Gamma \IndRcptl_\KPh(\pt) \iso \KTh_\KPh(\Cstar_\epsilon \Gamma)$.
    Moreover, the index map $\Ind^\Gamma \colon \Gamma \Klf_\KPh(\pt) \to \Gamma \IndRcptl_\KPh(\pt)$ recovers the assembly map $\mu \colon \KTh_\ast(\Bfree \Gamma) \to \KTh_\ast(\Cstar_\epsilon \Gamma)$ featuring in the strong Novikov conjecture.
  \end{ex}

  The external product in $\KTh$-homology also induces an external product,
  \begin{equation*}
    \Gamma \Klf_{n}(X) \tens \KTh_d^\lf(Y) \overset{\times}{\to} \Gamma \Klf_{n+d}(X \times Y).
  \end{equation*}

  \begin{prop}[Suspension isomorphism]\label{thm:suspensionIso}
    Let $Y$ be a proper metric space.
    There are isomorphisms $s$ and $\sigma$ which make the following diagram commutative,
    \begin{equation*}
      \begin{tikzcd}
        \Gamma \Klf_{\ast + 1}(Y \times \R) \dar{\iso}[swap]{s} \rar{\Ind^\Gamma} & \Gamma \IndRcptl_{\ast +1}(Y \times \R) \dar{\sigma}[swap]{\iso}  \\
        \Gamma \Klf_{\ast }(Y ) \rar{\Ind^\Gamma} & \Gamma \IndRcptl_{\ast}(Y)
      \end{tikzcd}
    \end{equation*}
    such that $s(x \times [\diracOp_{\R}]) = x$ for all $x \in \Gamma \Klf_{\ast }(Y )$.
  \end{prop}
  \begin{proof}
    To construct $s$ and $\sigma$ we use the Mayer--Vietoris boundary maps associated to the cover $Y \times \R = Y \times \Rgeq \cup Y \times \Rleq$ for $\KTh$-homology and for the $\KTh$-theory of the Roe algebra, respectively.\fxnote{Add reference for coarse Mayer--Vietoris}
    Indeed, take an admissible subset $Z \subseteq \Bfree \Gamma \times Y \times \R$ such that the cover $Z = Z \cap (\Bfree \Gamma \times Y \times \Rgeq) \cup Z \cap (\Bfree \Gamma \times Y \times \Rleq)$ is coarsely excisive so that we have a Mayer--Vietoris sequence both in $\KTh$-homology and for the $\KTh$-theory of the Roe algebra, see for example~\cite{HRY93coarse}.
    Let
    \begin{gather*}
      s_Z \colon \Klf_{\ast +1}(Z) \overset{\bdMV}\to \Klf_{\ast}(Z \cap (\Bfree \Gamma \times Y \times \{0\})) \to \Gamma \Klf_\ast(Y), \\
      \sigma_Z \colon \KTh_{\ast +1}(\roeAlg(\widetilde{Z})^\Gamma) \overset{\bdMV}\to \KTh_{\ast}(\roeAlg(\widetilde{Z} \cap (\Efree \Gamma \times Y \times \{0\}))^\Gamma) \to \Gamma \IndRcptl_\ast(Y).
    \end{gather*}
    The family of those admissible subset where the cover $Z = Z \cap (\Bfree \Gamma \times Y \times \Rgeq) \cup Z \cap (\Bfree \Gamma \times Y \times \Rleq)$ is coarsely excisive is cofinal in the directed set of all admissible subsets, hence the maps $s_Z$ and $\sigma_Z$ induce the required maps $s$ and $\sigma$ in the limit.
    Moreover, one can verify that the family of admissible $Z$ where $s_Z$ and $\sigma_Z$ are both defined and an isomorphism is also cofinal in the family of all admissible sets.
    The isomorphism statement relies on showing that we have a cofinal collection of admissible $Z$ such that $Z \cap (\Bfree \Gamma \times Y \times \Rgeq)$ and $Z \cap (\Bfree \Gamma \times Y \times \Rleq)$ is \emph{flasque}.
    A more detailed version of this argument can be found in~\cite[Proposition~4.2.3]{Zeidler16PhD}.

    Thus $s$ and $\sigma$ are isomorphisms.
    Finally, the claim $s(x \times [\diracOp_{\R}]) = x$ for all $x \in \Gamma \Klf_{\ast }(Y )$ is a standard fact in $\KTh$-homology which follows from $\bdMV([\diracOp_\R]) = 1$ for the Mayer--Vietoris boundary map associated to $\R = \Rgeq \cup \Rleq$.
  \end{proof}
  \begin{cor}\label{thm:locality}
    For every $\varepsilon > 0$, we have
    \begin{equation*}
      \Gamma \Klf_\ast(\R^q) \iso \colim_{K \subset \Bfree \Gamma} \Klf_\ast(K \times \R^q) \overset{\iota^!}{\iso} \colim_{K \subset \Bfree \Gamma} \Klf_\ast(K \times B_\varepsilon(0)),
    \end{equation*}
     where the colimit ranges over compact subsets $K \subseteq \Bfree \Gamma$ and the second isomorphism is induced by the inclusion of the open ball $\iota \colon B_\varepsilon(0) \hookrightarrow \R^q$.
  \end{cor}
  \begin{proof}
    Since for a compact subset $K \subseteq \Bfree \Gamma$ the set $K \times \R^q$ is admissible, we obtain a canonical map $J \colon \colim_{K \subset \Bfree \Gamma} \Klf_\ast(K \times \R^q) \to \Gamma \Klf_\ast(\R^q)$.
    The $q$-fold iteration of the suspension isomorphism from \cref{thm:suspensionIso} yields an isomorphism $s^q \colon \Gamma \Klf_\ast(\R^q) \iso \KTh_{\ast -q}(\Bfree \Gamma)$.
    An analogous argument as in the proof of \cref{thm:suspensionIso} produces an isomorphism $t^q \colon \colim_{K \subset \Bfree \Gamma}(K \times \R^q) \iso \KTh_{\ast -q}(\Bfree \Gamma)$ such that $t^q = s^q \circ J$.
    In particular, this shows that $J$  must be an isomorphism.

    For each $K \subseteq \Bfree \Gamma$, the restriction $\iota^{!} \colon \Klf_\ast(K \times \R^q) \to \Klf_\ast(K \times B_\varepsilon(0))$ is induced by the map on $K \times \R^q$ that is the identity on $K \times B_\varepsilon(0)$ and takes $K \times (\R^q \setminus B_\varepsilon(0))$ to infinity in the one-point compactification of $K \times B_\varepsilon(0)$.
    Since this map induces a homotopy equivalence between the one-point compactifications, this implies that $\iota^{!} \colon \Klf_\ast(K \times \R^q) \to \Klf_\ast(K \times B_\varepsilon(0))$ is an isomorphism.
  \end{proof}
    \Cref{thm:locality} implies that classes in $\Gamma \Klf_\ast(\R^q)$ (and thus their images in $\Gamma \IndRcptl_\ast(\R^q)$) depend only on the restrictions to arbitrarily small open subsets.
    A very similar localization property was exhibited by Schick--Zadeh~\cite{SZ13Multi} and is at the heart of their approach to the multi-partitioned manifold index theorem.
    Analogously, our approach to the theorem in the next subsection crucially relies on the localization property from \cref{thm:locality}.
    \subsection{Multi-partitioned manifolds}\fxnote{Maybe add some more explanation on how our approach relates to Siegel and Schick--Zadeh}
    Let $f \colon X \to Y$ be a proper map, $u \colon X \to \Bfree \Gamma$ classifying a covering $p \colon \tilde{X} \to X$.
  Then the map $u \times f \colon X \to \Bfree \Gamma \times Y$ induces a map $(u \times f)_\KPh \colon \KTh_\KPh^\lf(X) \to \Gamma \Klf_\ast(Y)$.
  If $f$ is also coarse, then the $\Gamma$-equivariant map $\tilde{u} \times (f \circ p) \colon \tilde{X} \to \Efree \Gamma \times Y$ induces a map $(\tilde{u} \times (f \circ p))_\KPh \colon \KTh_\KPh(\roeAlg(\tilde{X})^\Gamma) \to \Gamma \IndRcptl_\KPh(Y)$.    \begin{defi}
  A complete Riemannian manifold $X$ is called $q$-multi-partitioned by a closed submanifold $M \subseteq X$ via a continuous coarse map $f \colon X \to \R^q$ if $f$ is smooth near $f^{-1}(0)$ such that $0 \in \R^q$ is a regular value with $f^{-1}(0) = M$.
  \end{defi}
    \begin{defi}\label{defi:partitionedMfdIndex}
    Let $X$ be a complete spin $m$-manifold that is $q$-multi-partitioned by $M \subseteq X$ via $f \colon X \to \R^q$.
    Fix a $\Gamma$-covering $p \colon \tilde{X} \to X$ which is classified by a map $u \colon X \to \Bfree \Gamma$.
    Consider the lifted map $\tilde{u} \colon \tilde{X} \to \Efree \Gamma$.
    Then we define the \emph{higher partitioned manifold index} of $X$ to be
    \begin{equation*}
       \alpha_\PM^{f,u}(X) := (\tilde{u} \times (f\circ p))_\KPh(\Ind^\Gamma(\diracOp_{\tilde{X}})) \in \Gamma \IndRcptl_m(\R^q).
    \end{equation*}
  \end{defi}
         Furthermore, if $M$ is a closed spin manifold and $v \colon M \to \Bfree \Gamma$ a continuous map, then we set $\alpha^v(M) := \mu(v_\ast[\diracOp_M]) \in \KTh_\ast(\Cstar_\epsilon \Gamma)$, where $\mu \colon \KTh_\ast(\Bfree \Gamma) \to \KTh_\ast(\Cstar \Gamma)$ is the assembly map.
       If $v$ classifies the universal covering of $M$ this yields the Rosenberg index $\alpha^\Gamma(M)$.
  \begin{thm}[Multi-partitioned manifold index theorem]\label{thm:multiPartitioned}
   In the setup of \cref{defi:partitionedMfdIndex} we have
    \begin{equation*}
      \sigma^q \left( \alpha_\PM^{f,u}(X) \right) = \alpha^{u|_M}(M) \in \KTh_{m-q}(\Cstar_\epsilon \Gamma),
    \end{equation*}
    where $\sigma^q \colon \Gamma \IndRcptl_\ast(\R^q) \to \KTh_{\ast - q}(\Cstar_\epsilon \Gamma)$ is the $q$-fold iteration of the suspension isomorphism from \cref{thm:suspensionIso}.
  \end{thm}
  \begin{proof}
    We have $\sigma^q ( \alpha_\PM^{f,u}(X) ) = \Ind^\Gamma(s^q(u \times f)_\KPh([\diracOp_{X}]))$ by \cref{thm:suspensionIso}.
    We first deal with the product situation $X = M \times \R^q$ and $u = v \circ \pr_1$.
    In this special case we have $[\diracOp_{X}] =  [\diracOp_M] \times [\diracOp_{\R^q}]$ and the statement follows from an iterated application of the product formula from \cref{thm:suspensionIso}:
    \begin{equation*}
     \sigma^q \left( \alpha_\PM^{f,u}(X) \right) =  \Ind^\Gamma \left( s^q \left( v_\KPh([\diracOp_M]) \times [\diracOp_{\R^q}] \right) \right) = \Ind^\Gamma(v_\KPh[\diracOp_M]) = \alpha^{v}(M).
    \end{equation*}

    In the general case we may assume without loss of generality that there exists $\varepsilon > 0$ such that $f^{-1}(B_\varepsilon(0)) \iso M \times B_\varepsilon(0)$ isometrically.
    Furthermore, we consider the following commutative diagram where we set $v := u|_{M}$ and make extensive use of \cref{thm:suspensionIso,thm:locality}.
    \begin{equation*}
    \begin{tikzcd}
      \KTh^\lf_\KPh(X) \rar{(u \times f)_\KPh} \dar[swap]{\iota^!} & \Gamma \Klf_\KPh(\R^q) \dar{\iota^!}[swap]{\iso} \ar[bend left]{ddr}{s^q}[swap]{\iso} \\
       \KTh^\lf_\KPh(f^{-1}(B_\varepsilon(0))) \rar{(u \times f)_\KPh} & \colim_{K \subset \Bfree \Gamma}\KTh_{n}^\lf(K \times B_\varepsilon(0) ) \\
      \KTh^\lf_\KPh(M \times B_\varepsilon(0)) \uar[swap]{\iso} \ar{ur}[swap]{(v \times \id)_\KPh} & & \KTh_{\KPh-q}(\Bfree \Gamma)\\
      \KTh^\lf_\KPh(M \times \R^q) \uar{\iota^!}[swap]{\iso} \rar{(v \times \id)_\KPh} & \Gamma \Klf_\KPh(\R^q). \ar{uu}{\iso}[swap]{\iota^!} \ar{ur}{\iso}[swap]{s^q}
    \end{tikzcd}
    \end{equation*}
    Since $f^{-1}(B_\varepsilon(0)) \iso M \times B_\varepsilon(0)$, the class $[\diracOp_{X}] \in \Klf_m(X)$ goes to $[\diracOp_M] \times [\diracOp_{\R^q}] \in \Klf_m(M \times \R^q)$ following the left vertical maps in the diagram from top to bottom.
    Thus the diagram implies $(u \times f)_\KPh([\diracOp_{X}]) =v_\KPh([\diracOp_M]) \times [\diracOp_{\R^q}] \in \Gamma \Klf_m(\R^q)$.
    This reduces the general case to the product situation which has already been established.
  \end{proof}
  \begin{cor}
    If $\alpha^{u|_M}(M) \neq 0$ in the setup of \cref{defi:partitionedMfdIndex}, then $\Ind^\Gamma(\diracOp_{\tilde{X}}) \neq 0$.
    In this case the Riemannian metric on $X$ does not have uniform positive scalar curvature.
  \end{cor}
  \subsection{Fiber bundles and codimension one}
  We are now almost ready to prove \cref{thm:fibreBundleObstruction,thm:codim1Obstruction}.
  Before doing that, we state the result of \citeauthor{Dranishnikov06hypereuclidean} which is needed for \cref{thm:fibreBundleObstruction}.
  \begin{thm}[{\cite[Theorem~3.5]{Dranishnikov06hypereuclidean}}]\label{thm:DraHypereuclidean}
    Let $\tilde{B}$ be the universal covering of a closed aspherical $q$-manifold $B$ with $\asdim(\pi_1(B)) < \infty$.
    Then there exists $k \in \N$ and a proper Lipschitz map $g \colon \tilde{B} \times \R^k \to \R^{q+k}$ of degree $1$.
  \end{thm}
  \begin{proof}[Proof of \cref{thm:fibreBundleObstruction}]
    By \cref{thm:DraHypereuclidean}, we may assume that there exists a proper Lipschitz map $g \colon \tilde{B} \to \R^q$ of degree $1$ (if necessary, replace the entire bundle by its product with the $k$-torus $S^1 \times \cdots \times S^1$).
    Since Lipschitz functions can be approximated by smooth Lipschitz functions (see for example~\cite{GW79Approximation}), we may suppose without loss of generality that $g$ is smooth.
    In addition, we may assume that $0 \in \R^q$ is a regular value by Sard's theorem.
    Now consider the covering $\bar{M} \twoheadrightarrow M$ with $\pi_1(\bar{M}) = \Lambda = \pi_1(N)$.
    The bundle projection $\pi \colon M \to B$ lifts to a $\Gamma/\Lambda$-equivariant smooth map $\bar{\pi} \colon \bar{M} \to \tilde{B}$.
    Let $N^\prime := (g \circ \bar{\pi})^{-1}(0)$.
    Then $\bar{M}$ is $q$-multi-partitioned by $N^\prime$ via $f := g \circ \bar{\pi}$.
    Let $u \colon \bar{M} \to \Bfree \Lambda$ be the map that classifies the $\Lambda$-covering $p \colon \tilde{M} \to \bar{M}$, where $\tilde{M}$ is the universal covering of $M$.
    Since $g$ has degree $1$ and each fiber of $\bar{\pi}$ is a copy of $N$ inside $\bar{M}$ over each of which $p$ restricts to the universal covering, we have that $\alpha^{u|_{N^\prime}}(N^\prime) = \alpha^\Lambda(N) \in \KTh_{m-q}(\Cstar \Lambda)$.
    Now consider the homomorphism $\Psi \colon \KTh_\ast(\Cstar_\epsilon \Gamma) \to \KTh_{\ast-q}(\Cstar_\epsilon \Lambda)$ given by the following composition
    \begin{equation*}
      \Psi \colon \KTh_\ast(\Cstar_\epsilon \Gamma) \iso \KTh_\ast(\roeAlg(\tilde{M})^\Gamma) \to \KTh_\ast(\roeAlg(\tilde{M})^\Lambda) \overset{\tilde{u} \times (f \circ p)}{\to} \Lambda \IndRcptl_\ast(\R^q) \overset{\sigma^q}{\to} \KTh_{\ast-q}(\Cstar_\epsilon \Lambda),
    \end{equation*}
    where the second map is induced by the inclusion $\roeAlg(\tilde{M})^\Gamma \subseteq \roeAlg(\tilde{M})^\Lambda$ that just forgets part of the equivariance.
  We have
  \begin{equation*}
    \Psi(\alpha^\Gamma(M)) = \sigma^q \left(\alpha_\PM^{f, u}(\bar{M})  \right)= \alpha^{u|_{N^\prime}}(N^\prime) = \alpha^\Lambda(N),
  \end{equation*}
  where the first equality is by definition of $\alpha_\PM^{f, u}(\bar{M})$ and the second equality is due to \cref{thm:multiPartitioned} applied to $f = g \circ \bar{\pi} \colon \bar{M} \to \R^q$ and $u \colon \bar{M} \to \Bfree \Lambda$.
  Since $\Psi$ is a homomorphism this concludes the proof.
  \end{proof}
  \begin{proof}[Proof of \cref{thm:codim1Obstruction}]
    The following is very similar to the previous proof:
    We again consider the covering $\bar{M} \to M$ such that $\pi_1 \bar{M} = \Lambda$.
   With the right choice of basepoints it is possible to lift the inclusion $N \hookrightarrow M$ to an embedding $N \hookrightarrow \bar{M}$.
    Since $N \hookrightarrow \bar{M}$ has codimension one with trivial normal bundle and is an isomorphism on $\pi_1$, it follows that $\bar{M} \setminus N$ has precisely two connected components.
      Hence $\bar{M}$ is partitioned (or $1$-multi-partitioned in our terminology above) by $N$ via a map $f \colon \bar{M} \to \R$ which is essentially the distance function from $N$.
      Let $\tilde{M}$ be the universal covering of $M$ and $u \colon \bar{M} \to \Bfree \Lambda$ the map that classifies the $\Lambda$-covering $p \colon \tilde{M} \to \bar{M}$.
      Again we obtain a map
       \begin{equation*}
      \Psi \colon \KTh_\ast(\Cstar_\epsilon \Gamma) \iso \KTh_\ast(\roeAlg(\tilde{M})^\Gamma) \to \KTh_\ast(\roeAlg(\tilde{M})^\Lambda) \overset{\tilde{u} \times (f \circ p)}{\to} \Lambda \IndRcptl_\ast(\R) \overset{\sigma}{\to} \KTh_{\ast-1}(\Cstar_\epsilon \Lambda)
    \end{equation*}
    such that $\Psi(\alpha^\Gamma(M)) = \alpha^\Lambda(N)$.
  \end{proof}

  \section{Higher \texorpdfstring{$\Ahat$}{A-hat} obstructions via submanifolds}\label{sec:proofOfCohomological}
    \begin{setup} \label{defi:multipleSetup}
    In addition to \cref{defi:setup}, let $N = N_1 \cap \ldots \cap N_k$, where $N_1, \ldots, N_k \subseteq M$ are closed submanifolds with trivial normal bundle that intersect mutually transversely\footnote{To be precise, this means that the inclusion $N_1 \times \cdots \times N_k \hookrightarrow M^k$ is transverse to the diagonal embedding $\triangle \colon M \hookrightarrow M^k$ in the usual sense.}.
    Let $d$ be the maximum of the codimensions of the submanifolds $N_i$ for $i \in \{1, \ldots, k\}$.
    Denote by $u \colon M \to \Bfree \Gamma$ a classifying map of the universal covering and let $v := u \circ \iota  \colon N \to \Bfree \Gamma$.
    Moreover, let $w \colon N \to \Bfree \Lambda$ a classifying map of the universal covering of $N$.
  \end{setup}
  We follow the notation of \cite{HS08Low} and let $\Lambda^*(\Bfree \Gamma)$ denote the subring of $\homlgy^\ast(\Bfree \Gamma; \Q)$ generated by cohomology classes of degree at most two.
  \begin{prop}\label{thm:cohomologicalObstruction}
    Let $\epsilon \in \{ \redd, \max \}$.
    In \cref{defi:multipleSetup} suppose that the induced map in relative homology
    \begin{equation}
      (u, \id_N)_\ast \colon \homlgy_k(M, N) \to \homlgy_k(v) \quad \text{is injective for $2 \leq k \leq d$}.\label{eq:weakHomologicalCond}
    \end{equation}
    Assume furthermore that one of the following conditions holds:
    \begin{enumerate}[(a)]
          \item We have $\epsilon = \max$, $d \leq 2$ and there exists $x \in \Lambda^\ast(\Bfree \Gamma)$ such that the higher $\Ahat$-genus $\left< \Ahat(\tangentBdl N) \cup v^\ast(x), [N] \right>$ does not vanish.\label{item:cohomologicalCodim2}
      \item The group $\Gamma$ satifies the \parensup{maximal, if $\epsilon = \max$} strong Novikov conjecture and there exists $x \in \homlgy^\ast(\Bfree \Gamma; \Q)$ such that the higher $\Ahat$-genus $\left< \Ahat(\tangentBdl N) \cup v^\ast(x), [N] \right>$ does not vanish.\label{item:cohomologicalhigherCodim}
    \end{enumerate}
    Then $\alpha^\Gamma(M) \in \KTh_\ast(\Cstar_\epsilon \Gamma)$ does not vanish.
    In particular, $M$ does not admit a metric of positive scalar curvature.
  \end{prop}

   \begin{proof}
      Let $\eta_{i} \in \homlgy^{\ast}(M; \Q)$ denote the Poincaré dual of $N_i \subseteq M$.
      Since $N_i$ has trivial normal bundle the restriction of $\eta_i$ to $N_i$ vanishes.
      In particular, $\iota^\ast \eta_i = 0 \in \homlgy^\ast(N; \Q)$, so there exists $\tilde{\eta}_i \in \homlgy^\ast(M, N; \Q)$ that restricts to $\eta_i \in \homlgy^\ast(M; \Q)$.
            By the upper bound on the codimensions, the degree of $\eta_i$ is at most $d$ for each $i \in \{1, \ldots, k\}$.
            Note that $u \colon M \to \Bfree \Gamma$ is $2$-connected and thus $(u, \id_N)_\ast \colon \homlgy_1(M, N) \to \homlgy_1(v)$ is an isomorphism by the Hurewicz theorem and the long exact sequence associated to the triple $N \hookrightarrow M \overset{u}{\to} \Bfree \Gamma$.
      Together with \labelcref{eq:weakHomologicalCond} this implies that there exists $\tilde{\xi}_i \in \homlgy^\ast(v; \Q)$ such that $(u, \id_N)^\ast \tilde{\xi}_i = \tilde{\eta}_i$ for all $i \in \{1, \ldots, k\}$.
      Restricting these to $\Bfree \Gamma$, we get $\xi_i \in \homlgy^\ast(\Bfree \Gamma; \Q)$ such that $u^\ast \xi_i = \eta_i$.
      We have that $\eta = \eta_1 \cup \ldots \cup \eta_k = u^\ast(\xi)$ is the Poincaré dual of $N = N_1 \cap \ldots \cap N_k$, where $\xi := \xi_1 \cup \ldots \cup \xi_k$.
      For each $x \in \homlgy^\ast(\Bfree \Gamma; \Q)$, we then compute
      \begin{align*}
       \left< \Ahat(\tangentBdl N) \cup v^\ast(x), [N] \right> &= \left< \Ahat(\tangentBdl N) \cup \Ahat(\nu(N \hookrightarrow M)) \cup v^\ast(x), [N] \right> \\
      &= \left< \iota^* \Ahat(\tangentBdl M) \cup  v^\ast(x), [N] \right> \\
      &= \left< \Ahat(\tangentBdl M) \cup u^\ast(x) \cup \eta, [M] \right> \\
      &= \left< \Ahat(\tangentBdl M) \cup u^\ast(x \cup \xi), [M] \right> \\
       &=  \left< u^*(x \cup \xi), \chern([\diracOp_M]) \right>,
      \end{align*}
      where triviality of the normal bundle $\nu(N \hookrightarrow M)$ is used in the first equality.
      In other words, the particular higher $\Ahat$-genus of $N$ we started with can be rewritten as a higher $\Ahat$-genus of $M$.

      In case \labelcref{item:cohomologicalCodim2}, this implies that $\left< z, \chern(u_\ast [\diracOp_M]) \right> \neq 0$, where $z := x \cup \xi \in \Lambda^\ast(\Bfree \Gamma)$.
      Hence by~\cite[Theorem~1.2]{HS08Low}, this shows that $\alpha^\Gamma(M) = \mu(u_\ast([\diracOp_M])) \neq 0 \in \KTh_\ast(\Cstar_{\max} \Gamma)$.
      In case \labelcref{item:cohomologicalhigherCodim}, the computation simply shows that $0 \neq u_\ast([\diracOp_M]) \in \KTh_\KPh(\Bfree \Gamma) \tens \Q$.
      Hence by the postulated rational injectivity of the (maximal, if $\epsilon = \max$) assembly map, the higher index does not vanish.
  \end{proof}
  It remains to put forward some further (sufficient) conditions for the homological condition \labelcref{eq:weakHomologicalCond}.
  For instance, we find it conceptually appealing to consider the square
  \begin{equation*}
       \begin{tikzcd}
        N \rar{\iota} \dar{w} & M  \dar{u} \\
        \Bfree \Lambda \rar{j} & \Bfree \Gamma
      \end{tikzcd},
   \end{equation*}
   and ask the induced map in relative homology $\homlgy_\ast(M, N) \to \homlgy_\ast(\Bfree \Gamma, \Bfree \Lambda)$ to be an equivalence up to a certain degree.
   Indeed, as it turns out in the lemma below, this is an easy sufficient condition for \labelcref{eq:weakHomologicalCond}.
   Moreover, $\homlgy_\ast(M, N) \to \homlgy_\ast(\Bfree \Gamma, \Bfree \Lambda)$ being an isomorphism up to degree two and surjective in degree three is equivalent to surjectivity of $\pi_2(N) \to \pi_2(M)$.
   The latter is precisely the condition that we have already encountered in \cref{rem:HPSmoreGeneral}.

  \begin{lem}\label{thm:translatingHomotopyConditions}
    Suppose that in \cref{defi:multipleSetup} one of the following conditions holds.
    \begin{itemize}
      \item [(a)] \label{item:pi2surj}The map $\pi_2(N) \to \pi_2(M)$ is surjective and $d = 2$.
      \item [(a\textprime)]\label{item:homlgyEqu}The map $\homlgy_k(M, N) \to \homlgy_k(\Bfree \Gamma, \Bfree \Lambda)$ is an isomorphism for $2 \leq k \leq d$ and surjective for $k = d+1$.
      \item [(b)]The homotopy groups $\pi_k(M)$ vanish for $2 \leq k \leq d$.
    \end{itemize}
    Then the condition \labelcref{eq:weakHomologicalCond} from the statement of \cref{thm:cohomologicalObstruction} is satisfied.

    Moreover, for $d = 2$ the conditions (a) and (a\textprime) are equivalent.
  \end{lem}
  \begin{proof}
  We first show that for $d = 2$, (a) and (a\textprime) are equivalent.
  Indeed, consider the following diagram of homotopy cofiber sequences:
    \begin{equation*}
         \begin{tikzcd}
        N \rar{\iota} \dar{w} & M \rar \dar{u} & C_\iota \dar \\
        \Bfree \Lambda \rar{j} \dar & \Bfree \Gamma \rar \dar & C_j \dar \\
        C_w \rar  & C_u \rar & C
      \end{tikzcd}
    \end{equation*}
    Since $w$ and $u$ are $2$-connected by construction, it follows by the Hurewicz theorem that $\homlgy_k(C_w) = \homlgy_k(C_u) = 0$ for $k = 1,2$ as well as $\homlgy_3(C_w) \iso \pi_3(w)$ and $\homlgy_3(C_u) \iso \pi_3(u)$.
    In particular, looking at the lower horizontal sequence in the diagram we see that we always have $\homlgy_k(C) = 0$ for $k = 1,2$.
    Moreover, since $\Bfree \Gamma$ and $\Bfree \Lambda$ are aspherical, we have $\pi_3(u) \iso \pi_2(M)$ and $\pi_3(w) \iso \pi_2(N)$.
    Thus surjectivity of $\pi_2(N) \to \pi_2(M)$ is equivalent to surjectivity of $\pi_3(w) \iso \homlgy_3(C_w) \to \homlgy_3(C_u) \iso \pi_3(u)$ which, in turn, is equivalent to $\homlgy_3(C) = 0$ since we always have $H_2(C_w) = 0$.
    Finally, turning to the right vertical sequence of the diagram, the vanishing of $H_3(C)$ is equivalent to (a\textprime) for $d = 2$ (since we have always $\homlgy_k(C) = 0$ for $k=1,2$).

      To see that (a\textprime) implies \labelcref{eq:weakHomologicalCond}, we just note that the map $\homlgy_k(M, N) \to \homlgy_k(\Bfree \Gamma, \Bfree \Lambda)$ factors as $\homlgy_k(M, N) \to \homlgy_k(v) \to \homlgy_k(\Bfree \Gamma, \Bfree \Lambda)$.

  To see that (b) implies \labelcref{eq:weakHomologicalCond}, consider the long exact sequence of the triple $N \hookrightarrow M \overset{u}{\to} \Bfree \Gamma$.
  \begin{equation*}
    \cdots \to \homlgy_{k+1}(u) \to \homlgy_k(M, N) \to \homlgy_k(v) \to \homlgy_k(u) \to \cdots
  \end{equation*}
  If $\pi_k(M) = 0$ for $2 \leq k \leq d$, then $u \colon M \to \Bfree \Gamma$ is $(d+1)$-connected and hence $\homlgy_k(u) = 0$ for $k \leq d+1$.
  In particular, $\homlgy_k(M, N) \to \homlgy_k(v)$ is even an isomorphism for $k \leq d$.
  \end{proof}

  Finally, \cref{thm:codim2IntersectionObstruction,thm:generalIntersectionObstruction} follow immediately now by combining cases (a) and (b) from \cref{thm:cohomologicalObstruction} (applied to $x = 1 \in \homlgy^0(\Bfree \Gamma)$) with cases (a) and (b) from \cref{thm:translatingHomotopyConditions}, respectively.


%
\subsubsection*{Acknowledgements}
This article was written during the author's doctoral studies at the University of Göttingen, where he was supported by the German Research Foundation (DFG) through the Research Training Group 1493 \enquote{Mathematical structures in modern quantum physics.}

  \printbibliography
  \Addresses
  \listoffixmes
\end{document}